\documentclass[11pt]{amsart}
\usepackage[utf8]{inputenc}
\usepackage{amssymb, amsthm}
\usepackage[leqno]{amsmath}
\usepackage[foot]{amsaddr}
\pdfoutput=1 
\usepackage{enumerate}
\usepackage{graphicx}

\usepackage{prettyref,cite}
\usepackage[colorlinks,linkcolor=blue,citecolor=red]{hyperref}
\usepackage [a4paper, footskip=1cm, headheight = 16pt, top=3.5cm, bottom=3.5cm,  right=3.5cm,  left=3.5cm ]{geometry}
\newtheorem{theorem}{Theorem}[section]

\mathchardef\mh="2D

\newcounter{tbox}

\newcommand{\sta}[1]{\refstepcounter{tbox}\noindent{ \parbox{\textwidth}{\vspace*{0.3cm}(\thetbox) \emph{#1}\vspace*{0.3cm}}}}

\makeatletter
\newcommand{\otherlabel}[2]{\protected@edef\@currentlabel{#2}\label{#1}}
\makeatother

    \title[List-$k$-coloring $H$-free graphs for all $k>4$]{List-$k$-coloring $H$-free graphs for all $k>4$}
\author{Maria Chudnovsky$^{\ast \amalg}$}
\author{Sepehr Hajebi $^{\mathsection}$}
\author{Sophie Spirkl$^{\mathsection \parallel}$}

\address{$^{\ast}$Princeton University, Princeton, NJ, USA}
\address{$^{\mathsection}$Department of Combinatorics and Optimization, University of Waterloo, Waterloo, Ontario, Canada}
\address{$^{\amalg}$ Supported by NSF-EPSRC Grant DMS-2120644 and by AFOSR grant FA9550-22-1-0083.} 
\address{$^{\parallel}$ We acknowledge the support of the Natural Sciences and Engineering Research Council of Canada (NSERC), [funding reference number RGPIN-2020-03912].
Cette recherche a \'et\'e financ\'ee par le Conseil de recherches en sciences naturelles et en g\'enie du Canada (CRSNG), [num\'ero de r\'ef\'erence RGPIN-2020-03912]. This project was funded in part by the Government of Ontario. This research was conducted while Spirkl was an Alfred P. Sloan Fellow.}

\date{\today}

\allowdisplaybreaks

\begin{document}
\maketitle
\begin{abstract}
Given an integer $k>4$ and a graph $H$, we prove that, assuming \textsf{P}$\neq$\textsf{NP}, the \textsc{List-$k$-Coloring Problem} restricted to $H$-free graphs can be solved in polynomial time if and only if either every component of $H$ is a path on at most three vertices, or removing the isolated vertices of $H$ leaves an induced subgraph of the five-vertex path. In fact, the ``if'' implication holds for all $k\geq 1$.
\end{abstract} 
\section{Introduction}

Graphs in this paper have finite vertex sets, no loops an no parallel edges. Let $G=(V(G),E(G))$ be a graph.  An \textit{induced subgraph} of $G$ is the graph $G\setminus X$ for some $X\subseteq V(G)$, that is, the graph obtained from $G$ by removing the vertices in $X$. For $X\subseteq V(G)$, we use both $X$ and $G[X]$ to denote the subgraph of $G$ induced on $X$, which is the same as $G\setminus (V(G)\setminus X)$. We also say $G$ \textit{contains} a graph $H$ if $H$ is isomorphic to an induced subgraph of $G$; otherwise, we say $G$ is \textit{$H$-free}.

For an integer $k\geq 1$, we write $[k]=\{1,\ldots, k\}$. Given a graph $G$,  a \textit{proper $k$-coloring} of $G$ is a map $\varphi:V(G)\rightarrow [k] $ such that for every edge $uv\in E(G)$, we have $\varphi(u)\neq \varphi(v)$. A \textit{list-$k$-assignment} for $G$ is a map $L:V(G)\rightarrow 2^{[k]}$. Given a list-$k$-assignment $L$ for $G$, an \textit{$L$-coloring} of $G$ is a proper $k$-coloring $\varphi$ of $G$ such that $\varphi(v)\in L(v)$ for all $v\in V(G)$. The \textsc{$k$-Coloring Problem} is to decide, for a graph $G$, whether $G$ admits a $k$-coloring, and the \textsc{List-$k$-Coloring Problem} asks, for a graph $G$ and a list-$k$-assignment of $G$, whether $G$ admits an $L$-coloring.

The \textsc{List-$2$-Coloring Problem} can be solved in polynomial time via a reduction to \textsc{2SAT} \cite{2sat1}, whereas the \textsc{$3$-Coloring Problem} is famously known to be \textsf{NP}-hard \cite{karp}. In fact, the \textsc{$3$-Coloring Problem} remains \textsf{NP}-hard in the class of $H$-free graphs except possibly for some rather restricted choices of $H$: 

\begin{theorem}[Holyer \cite{ClawFree}, Kami\'{n}ski and Lozin\cite{CycleFree}]\label{thm:clawgirth}
      Let $H$ be a graph with at least one component which is not a path. Then the \textsc{$3$-Coloring Problem} restricted to $H$-free graphs is \textsf{NP}-hard.
\end{theorem}

The converse to Theorem~\ref{thm:clawgirth} is wide open. In general, for the \textsc{$k$-Coloring Problem}, no value of $k\geq 3$ is known for which the ``easy'' choices of $H$ are completely distinguished from the ``hard'' ones. The situation with the \textsc{List-$k$-Coloring Problem} was also the same until recently, when the last two authors together with Li \cite{L5CrP3} settled the case $k=5$. For integers $r,s\geq 1$, we denote by $rP_s$ the graph obtained from the disjoint union of $r$ copies of the $s$-vertex path, and we write $P_s$ instead of $1Ps$. For graphs $H_1,H_2$, we write $H_1+H_2$ to denote the disjoint union of $H_1$ and $H_2$ (see Figure~\ref{fig:easyoutcomes}):

\begin{theorem}[Hajebi, Li, Spirkl \cite{L5CrP3}]\label{thm:L5C}
     Suppose that \textsf{P}$\neq$\textsf{NP}. Let $H$ be a graph. Then the \textsc{List-$5$-Coloring Problem} restricted to $H$-free graphs can be solved in polynomial time if and only if for some integer $r\geq 1$, either $rP_3$ or $P_5+rP_1$ contains $H$.
\end{theorem}

\begin{figure}
    \centering
    \includegraphics{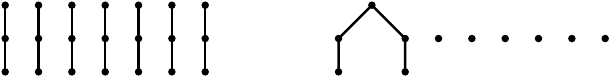}
    \caption{The graphs $7P_3$ (left) and $P_5+6P_1$ (right).}
    \label{fig:easyoutcomes}
\end{figure}

In this paper, we extend the conclusion of Theorem~\ref{thm:L5C} to all $k>4$. Indeed, like \cite{L5CrP3}, our main contribution is to show that for every $r\geq 1$, excluding (an induced subgraph graph of) $rP_3$ results in a polynomial-time solvable case, which also happens to be true for all $k\geq 1$:

\begin{theorem}\label{thm:main2}
    Let $k,r\geq 1$ be fixed integers. Then the  \textsc{List-$k$-Coloring Problem} restricted to $rP_3$-free graphs can be solved in polynomial time.
\end{theorem}

As shown in Theorem~\ref{thm:main} below, Theorem~\ref{thm:main2} along with a number of results from the literature (collected in Theorem~\ref{thm:knownpoly}) yields a full dichotomy for the \textsc{List-$k$-Coloring Problem} on $H$-free graphs for all $k>4$.

\begin{theorem}\label{thm:knownpoly}
Let $k\geq 1$ be integers. Then the \textsc{List-$k$-Coloring Problem} restricted to $H$-free graphs can be solved in polynomial time if 
\begin{itemize}
    \item $H=P_5+rP_1$ for some $r\geq 1$; \textup{(Couturier, Golovach, Kratsch and Paulusma \cite{LkP5+rP1&L5P4+P2-NP}). \label{thm:p5rp1}}
		\end{itemize}
  and remains \textsf{NP}-hard if either
		\begin{itemize}
			\item $H=P_6$ and $k>3$ \textup{(Golovach, Paulusma and Song \cite{L4P6-NP})\label{thm:listp6}}; or
			\item $H=P_4+P_2$ and $k>4$ \textup{(Couturier, Golovach, Kratsch and Paulusma \cite{LkP5+rP1&L5P4+P2-NP}) \label{thm:p4p2}}.
		\end{itemize}
	\end{theorem}

Thus, our main result is the following: 

\begin{theorem}\label{thm:main}
   Let $k>4$ be an integer and let  $H$ be a graph. Then the \textsc{List-$k$-Coloring Problem} restricted to $H$-free graphs can be solved in polynomial time if for some integer $r\geq 1$,  either $rP_3$ or $P_5+rP_1$ contains $H$.
     Otherwise, the \textsc{List-$k$-Coloring Problem} on $H$-free graphs is \textsf{NP}-hard.
\end{theorem}

\begin{proof}[Proof (assuming Theorem \ref{thm:main2}).]
	If $rP_3$ contains $H$ for some integer $r\geq 1$, then the result follows from Theorem~\ref{thm:main2}, and if $P_5+rP_1$ contains $H$ for some integer $r\geq 1$, then the result follows from the first bullet of Theorem \ref{thm:p5rp1}. So we may assume that neither holds. Our goal is then to show that the \textsc{List-$k$-Coloring Problem} on $H$-free graphs is \textsf{NP}-hard.
	
	By Theorem~\ref{thm:clawgirth}, we may assume that each component of $H$ is a path. Since  $rP_3$ does not contain $H$ for any $r\geq 1$, it follows that $H$ contains $P_4$. This, combined with the assumption that $P_5+rP_1$ does not contain $H$ for any $r\geq 1$, implies that $H$ contains either $P_6$ or $P_4+P_2$. But then the result follows from the second and the third bullet of Theorem \ref{thm:knownpoly}.
 \end{proof}

It remains to prove Theorem~\ref{thm:main2}, which we do in the next section.

\section{The algorithm}

We begin with providing some context. The proof of Theorem~\ref{thm:L5C} in \cite{L5CrP3} consists of two steps. The first one, which works for general $k$, reduces the problem in polynomial time to polynomially many instances in which no three vertices with a common color in their lists induce a path. The second step, confined to the case $k=5$, renders an intricate analysis within radius-two balls around vertices with list-size more than two, eventually reducing the problem to lists of size at most two (and so to \textsc{2SAT}). In our proof of Theorem~\ref{thm:main2}, the first step remains untouched, but the second step is superseded by Theorem~\ref{thm:matchingmagic} below, which has a significantly less technical proof, and holds true for all $k$.

For integers $k,r\geq 1$, by a \textit{$(k,r)$-instance} we mean a pair $(G,L)$ where $G$ is an $rP_3$-free graph and $L$ is a list-$k$-assignment of $G$. We say that   a $(k,r)$-instance $(G,L)$ is \textit{admissible} if $G$ admits an $L$-coloring.

\begin{theorem}\label{thm:matchingmagic}
    Let $k,r\geq 1$ be fixed integers. Let $(G,L)$ be a $(k,r)$-instance where $G$ has $n\geq 1$ vertices. Assume that for every $3$-subset $\{x,y,z\}$ of $V(G)$ inducing a path, we have $L(x)\cap L(y)\cap L(z)=\emptyset$ Then it can be decided in time $\mathcal{O}(n^{5/2})$ whether $(G,L)$ is admissible.
\end{theorem}
\begin{proof}
    For every $i\in [k]$, let $G_i=G[\{v\in V(G): i\in L(v)\}]$. By our assumption, we have that $G_i$ is $P_3$-free, and so it follows that every component of $G_i$ is a clique of $G$. Let $\mathcal{C}_i$ be the set of all components of $G_i$. 
    
  We construct a bipartite graph $\Gamma$ with bipartition $(A,B)$ where the vertices in $A$ and $B$ are labelled as
    $$A=\{a_v:v\in V(G)\};$$  
    $$B=\{b^i_C:i\in [k], C\in \mathcal{C}_i\}$$
    such that 
    $$E(\Gamma)=\bigcup_{i=1}^k\bigcup_{C\in \mathcal{C}_i}\{a_vb^i_{C}: v\in C\}.$$
    See Figure~\ref{fig:Gamma}. It follows that $|V(\Gamma)|=n+ |\mathcal{C}_1|+\cdots+|\mathcal{C}_k|\leq (k+1)n$. Moreover, we have:

    \begin{figure}
        \centering
        \includegraphics[scale=0.9]{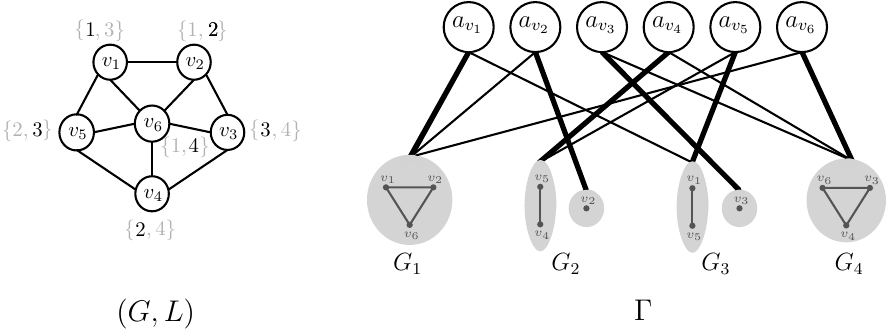}
        \caption{A $(4,2)$-instance $(G,L)$ (left) and the graph $\Gamma$ (right).}
        \label{fig:Gamma}
    \end{figure}

    \sta{\label{st:matchingmagic}$(G,L)$ is admissible if and only if $\Gamma$ has a matching which covers all vertices in $A$.}

    To see the ``only if'' implication, assume that $G$ admits an $L$-coloring $\varphi$. Then, for every vertex $v\in V(G)$, we have $v\in V(G_{\varphi(v)})$, and so there exists a unique component $C_v\in \mathcal{C}_{\varphi(v)}$ such that $v\in C_v$. This, along with the definition of $\Gamma$, implies that for every $v\in V(G)$, we have $a_vb^{\varphi(v)}_{C_v}\in E(\Gamma)$. Let $M=\{a_vb^{\varphi(v)}_{C_v}:v\in V(G)\}$; then we have $|M|=|A|$. We claim that $M$ is a matching in $\Gamma$. Clearly, no two edge in $M$ share an end in $A$. Also no two edges in $M$ share an end in $B$; for otherwise there are distinct vertices $u,v\in V(G)$ as well as $i\in [k]$ and $C\in \mathcal{C}_i$ such that $\varphi(u)=\varphi(v)=i$ and $C_u=C_v=C$. But this violates the fact that $\varphi$ is a proper coloring and $C$ is a clique of $G$. The claim follows, and so does the ``only if'' implication of \eqref{st:matchingmagic}.
    
    For the ``if'' implication, assume that there exists a matching $M\subseteq E(\Gamma)$ in $\Gamma$ which covers $A$. From the definition of $\Gamma$, it follows that there exists a map $\varphi:V(G)\rightarrow [k]$, as well as a component $C_v\in \mathcal{C}_{\varphi(v)}$ for each $v\in V(G)$, such that $M=\{a_vb^{\varphi(v)}_{C_v}:v\in V(G)\}$. We claim that $\varphi$ is an $L$-coloring of $G$.  Assume that $u,v\in V(G)$ are distinct and there exists $i\in [k]$ such that $\varphi(u)=\varphi(v)=i$. Then we have $C_u,C_v\in \mathcal{C}_i$. Also, since $M$ is a matching in $\Gamma$, it follows that $b^i_{C_u}$ and $b^i_{C_v}$ are distinct, which in turn implies that $C_u$ and $C_v$ are distinct components of $G_i$. This, combined with the fact that $G_i$ is an induced subgraph of $G$, implies that $u\in C_u$ and $v\in C_v$ are not adjacent in $G$. Thus, $\varphi$ is a proper $k$-coloring of $G$. Moreover, for every vertex $v\in V(G)$, since $a_vb^{\varphi(v)}_{C_v}\in M\subseteq E(\Gamma)$, it follows from the definition of $\Gamma$ that $v\in C_v\in \mathcal{C}_{\varphi(v)}$, and so $v\in V(G_{\varphi(v)})$. Therefore, we $\varphi(v)\in L(v)$ for every $v\in V(G)$. This proves \eqref{st:matchingmagic}.

    \medskip

    By \eqref{st:matchingmagic}, Theorem~\ref{thm:matchingmagic} is immediate from a well-known result of Hopcroft and Karp \cite{H&K} that the cardinality of the maximum matching in an $n$-vertex bipartite graph can be computed in time $\mathcal{O}(n^{5/2})$.
\end{proof}

Let us turn to the ``first step'' as discussed at the beginning of this section. For integers $k,r\geq 1$ and a $(k,r)$-instance $(G,L)$, by a \textit{$(G,L)$-profile} we mean a set $\mathcal{I}$ of pairs $(G',L')$ where $G'$ is an induced subgraph of $G$ and $L'$ is a list-$k$-assignment for $G'$ such that $L'(v)\subseteq L(v)$ for all $v\in V(G')$. In particular, if $\mathcal{I}$ is $(G,L)$-profile for a $(k,r)$-instance $(G,L)$, then every pair $(G',L')\in \mathcal{I}$ is a $(k,r)$-instance, as well.

\begin{theorem}[Hajebi, Li and Spirkl, see Theorem 5.1 in \cite{L5CrP3}]\label{thm:nogoodp3}
	Let $k,r\geq 1$ be fixed integers. Then there exists an integer $p=p(k,r)\geq 1$ such that for every $(k,r)$-instance $(G,L)$ with $|V(G)|=n\geq 1$, there is a $(G,L)$-profile $\mathcal{I}$ with the following specifications.
	\begin{itemize}
		\item $|\mathcal{I}|\leq \mathcal{O}\left(n^{p}\right)$ and $\mathcal{I}$ can be computed from $(G,L)$ in time $ \mathcal{O}\left(n^{p}\right)$.
		\item For every $(G',L')\in \mathcal{I}$ and every $3$-subset $\{x,y,z\}$ of $V(G')$ inducing a path, we have $L'(x)\cap L'(y)\cap L'(z)=\emptyset$.
		\item $(G,L)$ is admissible if and only if some $(G',L')\in \mathcal{I}$ is admissible. 
	\end{itemize}
\end{theorem} 

Finally, we merge Theorems~\ref{thm:matchingmagic} and \ref{thm:nogoodp3} to deduce Theorem~\ref{thm:main2}, restated as follows:

\begin{theorem}\label{thm:mainrP3}
    For all fixed integers $k,r\geq 1$, there exists an algorithm which, given a $(k,r)$-instance $(G,L)$, decides in polynomial time whether $(G,L)$ is admissible.
    \end{theorem}
\begin{proof}
    The algorithm is as follows. Given a $(k,r)$-instance $(G,L)$:
    \begin{enumerate}[\hspace{3mm} 1.]
        \item Compute the $(G,L)$-profile $\mathcal{I}$ as in Theorem~\ref{thm:nogoodp3}.
        \item For each $(k,r)$-instance $(G',L')\in \mathcal{I}$, decide whether $(G',L')$ is admissible.
        \item If there exists a $(k,r)$-instance $(G',L')\in \mathcal{I}$ which is admissible, then return ``$(G,L)$ is admissible.'' Otherwise, return ``$(G,L)$ is not admissible.''
    \end{enumerate}
    The correctness of the above algorithm is immediate from the third bullet of Theorem~\ref{thm:nogoodp3}. Also, the first two bullets of Theorem~\ref{thm:nogoodp3} combined with Theorem~\ref{thm:matchingmagic} imply that the above algorithm runs in polynomial time. This completes the proof of Theorem~\ref{thm:mainrP3}.
\end{proof}

	\bibliographystyle{abbrv}
	\bibliography{rP3ref}  
\end{document}